\documentclass[12pt,a4paper,reqno]{amsart}
\usepackage{amsmath,amssymb,amsthm,wasysym,calc,verbatim,enumitem,tikz,url,hyperref,mathrsfs,bbm,cite,fullpage}
\usetikzlibrary{shapes.misc,calc,intersections,patterns,decorations.pathreplacing, calligraphy}
\hypersetup{colorlinks=true,citecolor=blue,linktoc=page}
\hypersetup{%
    colorlinks,
    linkcolor={red!50!black},
    citecolor={green!50!black},
    urlcolor={blue!50!black}
}

\usepackage[abbrev,msc-links,backrefs]{amsrefs}
\usepackage{doi}

\renewcommand{\PrintDOI}[1]{\doi{#1}}

\AtBeginDocument{\def\MR#1{}}

\usepackage{tikz}
\usetikzlibrary{decorations.markings}
\usetikzlibrary{calc,positioning,decorations.pathmorphing,decorations.pathreplacing}
\usetikzlibrary{hobby,patterns}

\usepackage{fullpage}
\usepackage{setspace}
\usepackage{datetime}
\usepackage[margin=1cm]{caption}
\usepackage{graphicx}

\newtheorem{theorem}{Theorem}
\newtheorem{lemma}[theorem]{Lemma}
\newtheorem{proposition}[theorem]{Proposition}
\newtheorem{claim}[theorem]{Claim}
\newtheorem{corollary}[theorem]{Corollary}

\newtheorem{question}[theorem]{Question}

\theoremstyle{definition}
\newtheorem{definition}[theorem]{Definition}

\usepackage{enumitem}

\let\polishlcross=\l
\def\l{\ifmmode\ell\else\polishlcross\fi}

\makeatletter
\def\moverlay{\mathpalette\mov@rlay}
\def\mov@rlay#1#2{\leavevmode\vtop{    \baselineskip\z@skip\lineskiplimit-\maxdimen%
    \ialign{\hfil$\m@th#1##$\hfil\cr#2\crcr}}}
\newcommand{\charfusion}[3][\mathord]{
    #1{\ifx#1\mathop\vphantom{#2}\fi
        \mathpalette\mov@rlay{#2\cr#3}
      }
    \ifx#1\mathop\expandafter\displaylimits\fi}
\makeatother

\DeclareFontFamily{U}  {MnSymbolC}{}
\DeclareSymbolFont{MnSyC}         {U}  {MnSymbolC}{m}{n}
\DeclareFontShape{U}{MnSymbolC}{m}{n}{%
    <-6>  MnSymbolC5
   <6-7>  MnSymbolC6
   <7-8>  MnSymbolC7
   <8-9>  MnSymbolC8
   <9-10> MnSymbolC9
  <10-12> MnSymbolC10
  <12->   MnSymbolC12}{}
\DeclareMathSymbol{\powerset}{\mathord}{MnSyC}{180}

\makeatletter
\def\namedlabel#1#2{\begingroup
    #2%
    \def\@currentlabel{#2}%
    \phantomsection\label{#1}\endgroup
}
\makeatother

\numberwithin{theorem}{section}
\setlist[itemize]{leftmargin=1cm}
\setlist[enumerate]{leftmargin=1cm}

\everydisplay{\thickmuskip=7mu}
\renewcommand{\leq}{\leqslant}
\renewcommand{\geq}{\geqslant}

\let\epsilon\varepsilon%

\definecolor{green1}{rgb}{0.0, 0.5, 0.0}

\def\cJ{\mathcal{J}}

\def\EE{\mathbb{E}}

\def\NN{\mathbb{N}}
\def\PP{\mathbb{P}}
\def\QQ{\mathbb{Q}}
\def\RR{\mathbb{R}}

\def\bw{\mathbf{w}}

\def\1{\mathbbm{1}}
\def\<{\langle}
\def\>{\rangle}

\let\theta=\vartheta%
\let\phi=\varphi%

\DeclareMathOperator{\Sh}{Sh}
\DeclareMathOperator{\spread}{spread}
\DeclareMathOperator{\Spread}{Spread}

\pgfdeclaredecoration{complete sines}{initial}
{
    \state{initial}[
        width=+0pt,
        next state=sine,
        persistent precomputation={\pgfmathsetmacro\matchinglength{
            \pgfdecoratedinputsegmentlength / int(\pgfdecoratedinputsegmentlength/\pgfdecorationsegmentlength)}
            \setlength{\pgfdecorationsegmentlength}{\matchinglength pt}
        }] {}
    \state{sine}[width=\pgfdecorationsegmentlength]{
        \pgfpathsine{\pgfpoint{0.25\pgfdecorationsegmentlength}{0.5\pgfdecorationsegmentamplitude}}
        \pgfpathcosine{\pgfpoint{0.25\pgfdecorationsegmentlength}{-0.5\pgfdecorationsegmentamplitude}}
        \pgfpathsine{\pgfpoint{0.25\pgfdecorationsegmentlength}{-0.5\pgfdecorationsegmentamplitude}}
        \pgfpathcosine{\pgfpoint{0.25\pgfdecorationsegmentlength}{0.5\pgfdecorationsegmentamplitude}}
}
    \state{final}{}
}

\usepackage{cases}
\usepackage{eqparbox}

 \usepackage{lineno}
 \newcommand*\patchAmsMathEnvironmentForLineno[1]{
 \expandafter\let\csname old#1\expandafter\endcsname\csname #1\endcsname
 \expandafter\let\csname oldend#1\expandafter\endcsname\csname end#1\endcsname
 \renewenvironment{#1}
 {\linenomath\csname old#1\endcsname}
 {\csname oldend#1\endcsname\endlinenomath}}
 \newcommand*\patchBothAmsMathEnvironmentsForLineno[1]{
 \patchAmsMathEnvironmentForLineno{#1}
 \patchAmsMathEnvironmentForLineno{#1*}}
 \AtBeginDocument{
 \patchBothAmsMathEnvironmentsForLineno{equation}
 \patchBothAmsMathEnvironmentsForLineno{align}
 \patchBothAmsMathEnvironmentsForLineno{flalign}
 \patchBothAmsMathEnvironmentsForLineno{alignat}
 \patchBothAmsMathEnvironmentsForLineno{gather}
 \patchBothAmsMathEnvironmentsForLineno{multline}
 \patchAmsMathEnvironmentForLineno{numcases}
 }

\begin{document}
\onehalfspace%
\shortdate%
\yyyymmdddate%
\settimeformat{ampmtime}
\footskip=28pt

\title{Nowhere dense Ramsey Sets}

\author{Vojt\v{e}ch R\"{o}dl}

\author{Marcelo Sales}

\thanks{The authors were supported by NSF grant DMS 1764385. The first author was also supported by NSF grant DMS 2300347 and the second author by US Air Force grant FA9550-23-1-0298.}

\address{Department of Mathematics, Emory University, 
    Atlanta, GA, USA}
\email{vrodl@emory.edu}

\address{Department of Mathematics, University of California, 
    Irvine, CA, USA}
\email{mtsales@uci.edu}

\begin{abstract}
A set of points $S$ in Euclidean space $\RR^d$ is called \textit{Ramsey} if any finite coloring of $\RR^{\infty}$ yields a monochromatic copy of $S$. While characterization of Ramsey set remains a major open problem in Euclidean Ramsey theory, a stronger ``density'' result was considered in \cite{FR90}: If $S$ is a $d$-dimensional simplex, then for any $\mu>0$ there is an integer $d:=d(S,\mu)$ and finite configuration $X\subseteq \RR^d$ such that any subconfiguration $Y\subseteq X$ with $|Y|\geq \mu |X|$ contains a copy of $S$. Complementing this, here we show the existence of $\mu:=\mu(S)$ and of an infinite configuration $X\subseteq \RR^{\infty}$ with the property that any finite coloring of $X$ yields a monochromatic copy of $S$, yet for any finite set of points $Y\subseteq X$ contains a subset $Z\subseteq Y$ of size $|Z|\geq \mu |Y|$ without a copy of $S$.


\end{abstract}

\maketitle

\section{Introduction}\label{sec:intro}

We will find convenient to present our discussion in the framework of $\RR^\infty$, by which we understand a subspace of $\ell^2$ consisting of infinite sequences of real numbers with all but finitely many nonzero entries and with $\RR^\infty$ equipped by the usual euclidean metric. In other words, we can view $\RR^\infty$ as the infinite union $\RR^{\infty}=\bigcup_{d=1}^\infty \RR^d$, where we understand that the copies of $\RR^d$ are being included in one another.

For two configuration of points $A,B \subseteq \RR^\infty$ we will write 
\begin{align*}
    A\rightarrow (B)_r
\end{align*}
to denote the fact that any $r$-coloring of $A$ yields a monochromatic copy of $B$. By a (congruent) copy of $B$, we mean a subconfiguration $B'\subseteq A$ that is isometric to $B$, i.e., that exists a bijective map $\phi: B\rightarrow B'$ such that
\begin{align*}
    ||b_1-b_2||=||\phi(b_1)-\phi(b_2)||
\end{align*}
for every $b_1,b_2 \in B$. Given two configurations $A, B$ we say that $B$ is contained in $A$, and denote $B\subseteq A$, if there exists a copy $A'$ of $A$ such that $B\subseteq A'$ (in the set theoretical sense).

A finite configuration $S$ is said to be \textit{Ramsey} if $\RR^\infty \rightarrow (S)_r$ for every integer $r\geq 1$. The concept was introduced in \cite{EGMRSS73} by Erd\H{o}s, Graham, Montgomery, Rothschild, Spencer and Strauss, who proved that the vertex set of every brick (rectangular parallelepiped) of arbitrary finite dimension is Ramsey. The list of Ramsey configurations was extended by a few more configurations in \cites{FR90, FR86, K91, K92}. On the other hand, the authors of \cite{EGMRSS73} also proved that any Ramsey set is spherical, i.e., all points of $S$ lie on some finite dimensional sphere. They asked if the opposte implication is also true: If any spherical set is Ramsey. In \cite{Gr04} Ron Graham offered $\$1000$ dollars for deciding if this implication holds as well. Based on the evidence coming from known Ramsey configurations Leader, Russel and Walters \cite{LRW12} proposed an alternative conjecture. Calling a finite set transitive if its symmetry group is transitive, i.e., if all points play the same role, their conjecture states that Ramsey sets are precisely the transitive sets together with their subsets.

While the progress on these conjectures was very small, some alternative concepts were considered in \cites{Gr83, Gr85, MR95, FR90, FPRR18, CF18}. In this paper we will introduce another concept related to the Euclidean Ramsey problems. A $d$-dimensional simplex $S$ is a configuration consisting of $d+1$ affinely independent points in $\RR^{\infty}$. In \cite{FR90} it was proved that all simplices are Ramsey. One interesting feature of their proof is that they actually show the following stronger statement.

\begin{theorem}[\cite{FR90}]\label{thm:densitysimplex}
Let $S\subseteq \RR^{\infty}$ be a $d$-dimensional simplex and $0<\mu<1$ real number. Then there exists finite configuration $Y\subseteq \RR^{\infty}$ such that any subconfiguration $Z\subseteq Y$ of size $|Z|\geq \mu |Y|$ contains a copy of $S$.
\end{theorem}

In other words, Theorem \ref{thm:densitysimplex} not only finds a configuration $Y$ such that $Y\rightarrow (S)_r$, but also with the extra property that any subset of density at least $\mu$ contains a copy of $S$. One of the goals of this paper is to show an alternative proof of the fact that simplices are Ramsey where our set $Y$ does not have the density property. The following definition is central for our exposition.

\begin{definition}\label{def:pramsey}
A finite configuration $X\subseteq \RR^\infty$ is called \textit{P-Ramsey} if there exists a configuration $Y\subseteq \RR^{\infty}$ and a real number $\mu>0$ such that the following holds:
\begin{enumerate}
    \item[(i)] $Y\rightarrow (X)_r$ holds for every integer $r\geq 1$.
    \item[(ii)] Every finite subconfiguration $Y'\subseteq Y$ contains a configuration $Z\subseteq Y'$ with $|Z|\geq \mu|Y'|$ such that $Z$ is $X$-free
\end{enumerate}
\end{definition}

The definition is motivated by a similar problem concerning the equivalence of density and Ramsey statement for sum-free sets formulated by Pisier in \cite{G83} (see also \cites{ENR90, RRS22, NRS22}). Note that statement $(ii)$ of the P-Ramsey definition is in contrast with the density statement introduced in Theorem \ref{thm:densitysimplex}, since it says that every finite subconfiguration contains a large set without a copy of $X$. Hence, in some sense, the configuration $Y$ in Definition \ref{def:pramsey} is a witness that the Ramsey property is strictly weaker than the density property for $X$. 

Clearly, if $X$ is P-Ramsey, then $X$ is Ramsey. However, the converse is not so clear. In this paper, we start the study of P-Ramsey configurations by showing the following two results.

\begin{theorem}\label{thm:simplices}
All simplices are P-Ramsey.
\end{theorem}

We say that a configuration $B\subseteq \RR^{\infty}$ is a $d$-dimensional brick if there exists positive real numbers $a_1,\ldots, a_d\in \RR$ such that $B$ is congruent to the set
\begin{align*}
\left\{(x_1,\ldots,x_d):\:x_i=0 \text{ or }x_i=a_i,\, 1\leq i \leq d\right\}.
\end{align*}

\begin{theorem}\label{thm:bricks}
All bricks are P-Ramsey.
\end{theorem}

The paper is organized as follows: In Section \ref{sec:segments} we prove that segments are P-Ramsey. We introduce our main technical result in Section \ref{sec:robust}, a variation of the product theorem for P-Ramsey configurations, and prove Theorem \ref{thm:bricks}. Section \ref{sec:simplex} is devoted to the proof of Theorem \ref{thm:simplices}.

\section{Segments are P-Ramsey}\label{sec:segments}

We prove in this section that segments are P-Ramsey. In fact, we will prove a sligthly stronger statement. Recall that a weight vector $\bw: X\rightarrow [0,1]$ is \textit{stochastical} if $\sum_{x\in X}\bw(x)=1$.

\begin{lemma}\label{lem:segments}
Let $a,\gamma>0$ be real numbers. Then there exists a countable configuration $Y_a \subseteq \RR^{\infty}$ satisfying the following:
\begin{enumerate}
    \item[(i)] The set of squares of all distances of distinct points in $Y_a$ is
    \begin{align*}
        \left\{a^2, \frac{a^2}{1+\gamma+\gamma^2}, \frac{(1+\gamma^2)a^2}{1+\gamma+\gamma^2}, \frac{\gamma^2 a^2}{1+\gamma+\gamma^2}\right\}
    \end{align*}
    \item[(ii)] $Y_a\rightarrow (C)_r$ holds for every $r\geq 1$ and finite configuration $C\subseteq Y_a$. 
    \item[(iii)] For every finite subconfiguration $Y'\subseteq Y_a$ and stochastic weight vector $\bw:Y'\rightarrow [0,1]$, there exists a configuration $Z\subseteq Y'$ with no segments of lenght $a$ such that $\sum_{z\in Z}\bw(z)\geq \frac{1}{4}$.
    \item[(iv)] $Y_a$ does not contain an equilateral triangle of sides of lenght $a$.
\end{enumerate}
\end{lemma}

The proof of Lemma \ref{lem:segments} uses the following result about independent sets of a shift graph. The shift graph $\Sh(2, \NN)$ is the graph with vertex set $V(\Sh(2,\NN))=\NN^{(2)}$, i.e., the pairs of natural numbers, and edge set
\begin{align*}
    E(\Sh(2,\NN))=\left\{\left\{\{x,y\},\{y,z\}\right\}:\: x<y<z\right\}.
\end{align*}

\begin{claim}\label{cl:shift}
Let $\Sh(2,\NN)$ be the shift graph on the pairs of $\NN$. Then for every finite subset $X\subseteq V(\Sh(2,\NN))$ and stochastical weight vector $\bw:X\rightarrow [0,1]$, there exists an independent set $I\subseteq X$ such that
\begin{align*}
    \sum_{i\in I}\bw(i)\geq \frac{1}{4}.
\end{align*}
\end{claim}

\begin{proof}
Let $X\subseteq \NN^{(2)}$ be a finite subset of vertices of $\Sh(2,\NN)$ and let $\bw:X\rightarrow [0,1]$ be a stochastic weight vector. Consider a random coloring $c:\NN\rightarrow \{0,1\}$, where each integer $n$ is colored independently with probability
\begin{align*}
    \PP(c(n)=0)=\frac{1}{2}.
\end{align*}
Let $X_{0,1}$ be the random set defined by
\begin{align*}
    X_{0,1}=\left\{\{x,y\}\in X:\: x<y \text{ and } c(x)=0,\, c(y)=1\right\}.
\end{align*}
That is, $X_{0,1}$ are the ordered pairs of $X$ such that the first integer is of color $0$ and the last one of color $1$. One can see that $X_{0,1}$ is an independent set in $\Sh(2,\NN)$. Moreover, by letting
\begin{align*}
    Z_{0,1}=\sum_{x\in X_{0,1}}\bw(x)
\end{align*}
we have that 
\begin{align*}
    \EE(Z_{0,1})=\sum_{\{x,y\}\in X}\PP\big(\{c(x)=0\} \wedge \{c(y)=1\}\big)\bw(\{x,y\})=\sum_{\{x,y\}\in X}\frac{1}{4}\bw(\{x,y\})=\frac{1}{4}.
\end{align*}
Thus, by the first moment, with positive probability there is a coloring $c$ such that $X_{0,1}$ is an independent set satisfying the statement of the claim.
\end{proof}

\begin{proof}[Proof of Lemma \ref{lem:segments}]
Let $\{e_i\}_{i=1}^\infty$ be the standard basis of $\RR^\infty$. We construct a configuration $Y_a=\{y_e\}_{e\in \NN^{(2)}}\subseteq \RR^\infty$ by associating to each pair $e=\{i,j\} \in \NN^{(2)}$, $i<j$, the point
\begin{align*}
    y_e=\beta e_i -  \beta \gamma e_j,
\end{align*}
where $\beta=\frac{a}{\sqrt{2(1+\gamma+\gamma^2)}}$. We claim that the configuration $Y_a$ satisfies properties $(i)$, $(ii)$, $(iii)$ and $(iv)$ of Lemma \ref{lem:segments}.

Property $(i)$ comes from the fact that given two pairs $e=\{i,j\}$, $e'=\{i',j'\} \in \NN^{(2)}$ the square of the distance between $y_e$ and $y_{e'}$ can assume the following values
\begin{align*}
    ||y_e-y_{e'}||^2=\begin{cases}
    2\beta^2\gamma^2, \quad & \text{if }i=i'\\
    2\beta^2, \quad & \text{if }j=j'\\
    2\beta^2(1+\gamma+\gamma^2), \quad & \text{if }i=j' \text{ or } i'=j\\
    2\beta^2(1+\gamma^2), \quad & \text{if } \{i,j\}\cap\{i',j'\}=\emptyset.
    \end{cases}
\end{align*}
By plugging $\beta=\frac{a}{\sqrt{2(1+\gamma+\gamma^2)}}$ we obtain the set of distances of the statement. Moreover, another important consequence of the computation is that $||y_e-y_{e'}||=a$ if and only if $e$ and $e'$ are adjacent in $\Sh(2,\NN)$.

In order to prove $(ii)$, consider a finite configuration $C\subseteq Y_a$. Naturally $C$ can be written as $C=\{y_e\}_{e\in E}$ for some $E\subseteq \NN^{(2)}$. Since $E$ is finite, there exists an integer $n$ such that $E\subseteq [n]^{(2)}$. An $r$-coloring of $Y_a$ corresponds to an $r$-coloring of $\NN^{(2)}$. By Ramsey theorem \cite{R29}, there exists a set $W\subseteq \NN$ of size $n$ such that $W^{(2)}$ is monochromatic. Hence, this configuration $C'=\{y_e\}_{e\in W^{(2)}}$ is monochromatic. This implies property $(ii)$, since $C'$ contains a copy of $C$.

To check property $(iii)$, let $Y'\subseteq Y_a$ be a finite subconfiguration of $Y_a$. By our construction, this corresponds to a finite set $X\subseteq V(\Sh(2,\NN))$. Let $\bw':X\rightarrow [0,1]$ be the stochastic weight vector given by $\bw'(x)=\bw(y)$, where $y\in Y'$ is the corresponding point to $x\in X$. Claim \ref{cl:shift} applied to the vector $\bw'$ gives us an independent set $I\subseteq X$ in $\Sh(2,\NN)$ such that $\sum_{i\in I}\bw'(i)\geq \frac{1}{4}$. This corresponds to a subconfiguration $Z\subseteq Y'$ with no segments of length $a$ and such that
\begin{align*}
    \sum_{z\in Z}\bw(z)\geq \frac{1}{4}.
\end{align*}

Finally, property $(iv)$ follows from the fact that an equilateral triangle of sides of length $a$ corresponds to a triangle in $\Sh(2,\NN)$ and $\Sh(2,\NN)$ is triangle free.
\end{proof}

\section{Robust configurations}\label{sec:robust}

One of the main techniques developed in \cite{EGMRSS73} to prove that a configuration is Ramsey is the product theorem. The next statement follows directly from their work (see Theorem 20, \cite{EGMRSS73})

\begin{theorem}[\cite{EGMRSS73}]\label{thm:product}
Let $A$ and $B$ be finite configurations which are Ramsey and $X,Y\subseteq \RR^{\infty}$ be such that $X\rightarrow (A)_r$ and $Y\rightarrow (B)_r$ for every $r\geq 1$. Then $X\times Y \rightarrow (C)_r$ for $C\subseteq A\times B$ for every $r\geq 1$.
\end{theorem}

Unfortunately, it is not clear if a similar statement holds for P-Ramsey configurations. The goal of this sections is to develop a partial version of the product theorem that will enable us to prove Theorems \ref{thm:simplices} and \ref{thm:bricks}.

\begin{definition}\label{def:robust}
We say that a countable configuration $Y$ is \textit{robust} if for every finite configuration $C$ with $C\subseteq Y$ we have that $Y\rightarrow (C)_r$ for every $r\geq 1$.
\end{definition}

Note for instance, that by property $(ii)$ of Lemma \ref{lem:segments} we have the following.

\begin{corollary}\label{cor:Y_arobust}
Let $Y_a$ be the configuration obtained by Lemma \ref{lem:segments}. Then $Y_a$ is a robust configuration.
\end{corollary}

The following is our main result in the section. Recall that by $B\subseteq A$ we understand that there exists $B'$ copy of $B$ such that $B'\subseteq A$.

\begin{theorem}\label{thm:robust}
Let $B$ be a brick and $Y$ be a robust configuration. If $F\subseteq B\times Y$ and $F\not\subseteq Y$, then $F$ is P-Ramsey.
\end{theorem}

Theorem \ref{thm:robust} is a consequence of the following lemma.

\begin{lemma}\label{lem:segrobust}
Let $Y$ be a robust configuration, $A$ be a segment and $F$ a finite configuration with $|F|>1$. Then the following holds:
\begin{enumerate}
    \item[(a)] If $F\subseteq A\times Y$ and $F\not\subseteq Y$, then $F$ is P-Ramsey.
    \item[(b)] If $F\not\subseteq A\times Y$, then there exists a robust configuration $\tilde{Y}$ such that $A\times Y\subseteq \tilde{Y}$ and $F\not\subseteq \tilde{Y}$.
\end{enumerate}
\end{lemma}

\begin{proof}
Let $a$ be the length of the segment $A$, also let $D_Y$ be the set of all distances in $Y$ and let $D_F$ be the set of all distances in $F$. Consider the field extension $L=\QQ(a,D_Y,D_F)$ of $\QQ$, where $\QQ(a, D_Y, D_F)$ is the minimal field containing $a$, $D_Y$, $D_F$ and $\QQ$. Since $D_Y\cup D_F\cup \{a\}$ is countable, we have that $L$ is a countable extension of $\QQ$ and consequently $L\neq \RR$. Let $\gamma\in \RR$ be a transcedental number over $L$, i.e., 
\begin{align}\label{ceq:transcedental}
\text{there is no polynomial $p\in L[x]$ such that $p(\gamma)=0$.}
\end{align}

Let $Y_a$ be the configuration obtained by Lemma \ref{lem:segments} with parameters $a$ and $\gamma$. By property $(i)$ the set of all square distances is given by 
\begin{align*}
    \left\{a^2,\, \frac{a^2}{1+\gamma+\gamma^2},\, \frac{(1+\gamma^2)a^2}{1+\gamma+\gamma^2},\, \frac{\gamma^2 a^2}{1+\gamma+\gamma^2}\right\}.
\end{align*}
Note that while $a^2\in L$, due to the fact that $\gamma$ is transcedental, the other three distances are not in $L$. Indeed, to illustrate, assume for example that $\frac{a^2}{1+\gamma+\gamma^2} \in L$. Then there exists $b\in L$ such that
\begin{align*}
    \frac{a^2}{1+\gamma+\gamma^2}=b.
\end{align*}
This implies that $\gamma$ is a root of the polynomial $p\in L[x]$ given by $p(x)=bx^2+bx+b-a^2$, which contradicts the assumption that $\gamma$ is transcedental over $L$.

Before we address statements $(a)$ and $(b)$ of Lemma \ref{lem:segrobust}, we will prove the following claim. Let $\pi_A:Y_a\times Y\rightarrow Y_a$ and $\pi_Y: Y_a\times Y\rightarrow Y$ be the projection maps of $Y_a\times Y$ on $Y_a$ and $Y$, respectively.

\begin{claim}\label{cl:projection}
Let $F\subseteq Y_a\times Y$. Then either $F\subseteq Y$ or $\pi_A(F)$ is a copy of $A$.
\end{claim}

\begin{proof}
If $\pi_A(F)$ is a single point, then $F\subseteq Y$ and there is nothing to do. Thus, we may assume that $|\pi_A(F)|\geq 2$. Let $p,q$ be two points of $F$ such that $p'=\pi_A(p)$ and $q'=\pi_A(q)$ are distinct. We claim that $||p'-q'||=a$. Let $p''=\pi_Y(p)$ and $q''=\pi_Y(q)$. Since all distances from points of $F$ and $Y$ are in $L$, we have that $||p-q||^2, ||p''-q''||^2 \in L$. Thus, by Pythagoras theorem we have
\begin{align}\label{ceq:distance}
    ||p'-q'||^2=||p-q||^2-||p''-q''||^2 \in L.
\end{align}
On the other hand, by Lemma \ref{lem:segments} we have that
\begin{align*}
    ||p'-q'||^2\in \left\{a^2,\, \frac{a^2}{1+\gamma+\gamma^2},\, \frac{(1+\gamma^2)a^2}{1+\gamma+\gamma^2},\, \frac{\gamma^2 a^2}{1+\gamma+\gamma^2}\right\}.
\end{align*}
Due to our choice of $\gamma$, the value $a^2$ is the only one of the four values in the field $L$. Hence, due to (\ref{ceq:distance}) we have $||p'-q'||=a$.

Suppose that $|\pi_A(F)|\geq 3$. Then by the previous paragraph, there is an equilateral triangle of sides of length $a$ in $Y_a$, which contradicts property $(iv)$ of Lemma \ref{lem:segments}. Therefore, $\pi_A(F)$ is a segment of length $a$.
\end{proof}

Now we prove statement $(a)$ of Lemma \ref{lem:segrobust}. Let $F$ be a finite configuration, $|F|>1$, such that $F\subseteq A\times Y$ and $F\not\subseteq Y$ for a segment $A$ and a robust configuration $Y$. By Corollary \ref{cor:Y_arobust}, the configuration $Y_a$ is robust, wher $Y_a$ is defined with parameters $a$ and $\gamma$ satisfying (\ref{ceq:transcedental}). We claim that $Y_a\times Y$ testifies that $F$ is P-Ramsey.

To check property $(i)$ of Definition \ref{def:pramsey} we note that since $F\subseteq A\times Y$, there exists a finite configuration $C$ such that $F\subseteq A\times C$. Because $Y$ is robust, then $Y\rightarrow (C)_r$ for every $r\geq 1$. Lemma \ref{lem:segments} gives us that $Y_a\rightarrow (A)_r$ for every $r\geq 1$. Thus, by Theorem \ref{thm:product}, we have that $Y_a\times Y\rightarrow (F)_r$ for every $r\geq 1$.

In order to prove property $(ii)$ of Definition \ref{def:pramsey}, let $V\subseteq Y_a\times Y$ be a finite subconfiguration. Since $V$ is finite, there exists a finite subconfiguration $X\subseteq Y_a$ such that $V\subseteq X\times Y$. We partition $V$ into $V=\bigcup_{x\in X}V_x$ where $V_x=\pi_A^{-1}(x)$ are the elements of $V$ that projects to the point $x$ on $Y_a$. Let $\bw:X\rightarrow [0,1]$ be the stochastic weight vector defined by
\begin{align*}
    \bw(x)=\frac{|V_x|}{|V|}.
\end{align*}
By property $(iii)$ of Lemma \ref{lem:segments}, there exists a subconfiguration $Z\subseteq X$ with no segments of length $a$ such that
\begin{align}\label{ceq:eq1}
    \sum_{z\in Z} \bw(z)\geq \frac{1}{4}. 
\end{align}
Consider the configuration $U=\bigcup_{z\in Z} V_z$. We claim that $U$ does not contain a copy of $F$. Suppose to the contrary that $F\subseteq U$. Since $F\not\subseteq Y$, then by Claim \ref{cl:projection} the projection $\pi_A(U)$ contains a segment of length $a$. However, $\pi_A(U)=Z$, which contains no segment of length $a$, yielding a contradiction. Moreover, by (\ref{ceq:eq1})
\begin{align*}
|U|=\sum_{z\in Z}|V_z|=\sum_{z\in Z}|V|\bw(z)\geq \frac{1}{4}|V|,
\end{align*}
which proves property $(ii)$ of Definition \ref{def:pramsey} with $\mu=\frac{1}{4}$. Hence, $F$ is P-Ramsey.

Now we prove statement $(b)$ of Lemma \ref{lem:segrobust}. Suppose that $F\not\subseteq A\times Y$. We claim that $\tilde{Y}=Y_a\times Y$ is a robust configuration such that $F\not \subseteq Y_a\times Y$. We first show that $\tilde{Y}$ is robust. If $C\subseteq \tilde{Y}$ is a finite configuration, then there exist finite configurations $W_A\subseteq Y_a$ and $W\subseteq Y$ such that $C\subseteq W_A\times W$. Since $Y_a$ and $Y$ are robust, we have that $Y_a\rightarrow (W_A)_r$ and $Y\rightarrow (W)_r$ for every $r\geq 1$. By Theorem \ref{thm:product}, we obtain that $\tilde{Y}=Y_a\times Y \rightarrow (C)_r$, which proves that $\tilde{Y}$ is robust.

Assume by contradiction that $F\subseteq Y_a\times Y$. Then by Claim \ref{cl:projection}, we either have that $F\subseteq Y$ or $\pi_A(F)$ is a copy of $A$. In both cases, we have that $F\subseteq A\times Y$, which contradicts the hypothesis.
\end{proof}

We are now able to prove Theorem \ref{thm:robust}.

\begin{proof}[Proof of Theorem \ref{thm:robust}]
Let $B$ be a $d$-dimensional brick and let $Y$ be a given robust configuration. We will write $B=A_1\times\ldots \times A_d$ where $A_i$ is a segment. By the hypothesis of Theorem \ref{thm:robust} we are also given $F$ satisfying $F\subseteq B\times Y$ and $F\not\subseteq Y$. Our goal is to prove that $F$ is P-Ramsey. For that we will repeteadly apply Lemma \ref{lem:segrobust}. We will construct a sequence $Y_0,\ldots, Y_{\ell}$ of robust configurations with the property that $F\not \subseteq Y_i$, for $0\leq i \leq\ell$, as follows: Let $Y_0=Y$. Suppose that we already constructed $Y_0,\ldots, Y_i$. If $F\subseteq A_{i+1}\times Y_i$, then we stop the process and set $\ell=i$. Otherwise, by statement $(b)$ of Lemma \ref{lem:segrobust}, there exists a robust configuration $\tilde{Y}$ such that $A_{i+1}\times Y_i \subseteq \tilde{Y}$ and $F\not\subseteq \tilde{Y}$. Set $Y_{i+1}=\tilde{Y}$. A simple induction shows that for every $1\leq i \leq \ell$
\begin{align*}
    A_1\times \ldots \times A_i\times Y\subseteq Y_i.
\end{align*}
Since $F\subseteq B\times Y=A_1\times \ldots \times A_d \times Y$, the process terminates before the $d$-th step of the construction, i.e., $\ell<d$. This implies that $F\subseteq A_{\ell+1}\times Y_\ell$ and $F\not\subseteq Y_\ell$ and by statement $(a)$ of Lemma \ref{lem:segrobust}, we have that $F$ is P-Ramsey.
\end{proof}

A corollary of Theorem \ref{thm:robust} is that bricks are P-Ramsey. In fact, we prove the following slightly stronger statement that in particular implies Theorem \ref{thm:bricks}.

\begin{corollary}\label{cor:subbricks}
Let $B$ be a brick and $F\subseteq B$ be a subconfiguration with $|F|>1$. Then $F$ is P-Ramsey.
\end{corollary}

\begin{proof}
Suppose that $B$ is $d$-dimensional brick and write $B=A_1\times \ldots \times A_d$, where $A_i$ is a segment of length $a_i$ and $a_1\geq \ldots \geq a_d$. Let $\gamma>0$ be an arbitrary real number and let $Y_{A_d}$ be the configuration obtained by Lemma \ref{lem:segments} with parameters $\gamma$ and $a_d$. Suppose that $F\subseteq Y_{A_d}$. By the minimality of the segment $A_d$, we have that the minimum distance between two points in $F$ is at least $a_d$. Moreover, by property $(i)$ of Lemma \ref{lem:segments}, the diameter of $Y_{A_d}$ is exactly $a_d$. Hence, any two points of $F$ has distance $a_d$. If $|F|\geq 3$, then $Y_{A_d}$ contains an equilateral triangle of sides $a_d$. This contradicts property $(iv)$ of Lemma \ref{lem:segments}. Thus, $F$ is a copy of the segment $A_d$ and in this case $F$ is P-Ramsey by property $(ii)$ and $(iii)$ of Lemma \ref{lem:segments}.

Now suppose that $F\not \subseteq Y_{A_d}$. Since $A_d\subseteq Y_{A_d}$, then $F\subseteq A_1\times \ldots \times A_d\subseteq A_1\times \ldots \times A_{d-1}\times Y_{A_d}$. Therefore, $F$ satisfies the hypothesis of Theorem \ref{thm:robust} and we obtain that $F$ is P-Ramsey.
\end{proof}

\section{Simplices are P-Ramsey}\label{sec:simplex}

In this section we prove Theorem \ref{thm:simplices}. The proof follows the ideas from \cites{FR90, MR95}. First, we will introduce the terminology and auxiliary results from those papers. The main idea will be to prove that any simplex $S$ can be embedded in a product $B\times Y$, where $B$ is a brick and $Y$ is a robust configuration.

To address the robust configuration consider the following definition: Let $\{e_i\}_{i\geq 1}$ be the standard basis of $\RR^{\infty}$. Given an integer $k$, a vector $c=(c_1,\ldots,c_k) \in \RR^k$ and a $k$-tuple $J=\{j_1,\ldots,j_k\}\in \NN^{(k)}$ with $j_1<\ldots<j_k$, we define the point $\spread(c,J) \in \RR^{\infty}$ as
\begin{align*}
    \spread(c,J)=\sum_{\ell=1}^k c_\ell e_{j_\ell}.
\end{align*}
Given a subset of integers $A\subseteq \NN$, one can then define the configuration $\Spread(c,A)$ as
\begin{align*}
    \Spread(c,A)=\{\spread(c,J):\: J\in A^{(k)}\},
\end{align*}
where $A^{(k)}$ is the set of $k$-tuples of $A$.

The reason why spread configurations are interesting for us is twofold: One is that those configurations approximate simplices very well. The second is that it fits well in the context of P-Ramseyeness (see Claim \ref{cl:spreadrobust} below). The next result was proven in \cite{MR95}. Given real number $\rho>0$, we denote by $S_{\rho}(\RR^{\infty})$ the sphere of radius $\rho$ in $\RR^{\infty}$. For a linear subspace $Z\subseteq \RR^\infty$, let $S_{\rho}(Z)=S_{\rho}(\RR^\infty)\cap Z$.

\begin{proposition}[\cite{MR95}]\label{prop:spreadapprox}
For every $\delta>0$ and every integer $m$, there exist integer $n$, $k$, a $k$-dimensional vector $c=(c_1,\ldots,c_k)\in \RR^{k}$ with $||c||=\rho$ and an $m$-dimensional subspace $Z\subseteq \RR^{\infty}$ such that the following holds. For every $z\in S_{\rho}(Z)$, there is a point $y\in \Spread(c,[n])$ such that $||z-y||<\delta$.
\end{proposition}

Since any $d$-dimensional simplex can be embedded in any $d$-dimensional vector space, we obtain the following corollary from Proposition \ref{prop:spreadapprox}.

\begin{corollary}\label{cor:spread}
For $\delta<\rho/2$ and $d$-dimensional simplex $S=\{y_0,\ldots,y_d\}$ of circumradius $\rho(S)=\rho$, there exist integers $n$, $k$, a $k$-dimensional vector $c=(c_1,\ldots,c_k) \in \RR^k$ with $||c||=\rho$ and a $d$-dimensional simplex $S'=\{z_0,\ldots,z_d\}\subseteq \Spread(c,[n])$ such that $||y_i-z_i||<\delta$ for $0\leq i \leq d$.
\end{corollary}

The second reason is that Spread configurations are robust.

\begin{claim}\label{cl:spreadrobust}
$\Spread(c,\NN)$ is robust.
\end{claim}

\begin{proof}
Let $X\subseteq \Spread(c,\NN)$ be a finite configuration. Then there exist $N$ and $\cJ=\{J_1,\ldots,J_t\} \subseteq [N]^{(k)}$ such that $X=\{\spread(c,J): J\in \cJ\}$. Note that there exists a bijective map $\phi$ from $\Spread(c,\NN)$ to $\NN^{(k)}$ given by $\phi(\spread(c,J))=J$. Thus, for any finite coloring of $\Spread(c, \NN)$ there is a corresponding coloring of $\NN^{(k)}$. By Ramsey's theorem, there exist set $A\subseteq \NN$ of size $N$ such that $A^{(k)}$ is monochromatic. Therefore, the configuration $\Spread(c,A)$ is monochromatic. The result follows now since $X\subseteq \Spread(c,A)$. 
\end{proof}

Another important result for our proof is the next characterization of configurations of points in an Euclidean space. Let $M=(m_{ij})_{0\leq i,j\leq d}$ be a symmetric real matrix with zero entries on the main diagonal. We say that the matrix $M$ is of \textit{negative type} if
\begin{align}\label{ceq:eq2}
    \sum_{0\leq i<j \leq d}m_{ij}\lambda_i \lambda_j\leq 0
\end{align}
holds for all choices of $\lambda_0,\ldots,\lambda_d$ with $\lambda_0+\ldots+\lambda_d=0$ and $\lambda_0^2+\ldots+\lambda_d^2=1.$

\begin{theorem}[\cite{S38}]\label{thm:realization}
A finite configuration $X=\{x_0,\ldots,x_d\}$ with distances $d_{ij}=||x_i-x_j||$ can be embedded in the Euclidean space $\RR^d$ if and only if the matrix $M=(m_{ij})_{0\leq i, j\leq d}$ given by $m_{ij}=d_{ij}^2$ is of negative type. Moreover, $X$ is a $d$-dimensional simplex if and only if the inequality in (\ref{ceq:eq2}) is strict for all choices of $\lambda_0,\ldots,\lambda_d$.
\end{theorem}

Dekster and Wilker provided in \cite{DW87} a characterization of when the edge lengths of a simplex are realizable in the Euclidean space. As a consequence of Theorem \ref{thm:realization}, we show a weaker version of this result that is enough for the purpose of the paper.

\begin{corollary}\label{cor:almostsimplex}
Let $d$ be an integer and $\beta, \epsilon >0$ be real numbers such that $\epsilon<\frac{\beta}{64d^2}$. For any symmetric matrix of distances $D=\{d_{ij}\}_{0\leq i,j\leq d}$ satisfying
\begin{align}\label{ceq:approximate}
    \beta-\epsilon\leq d_{ij} \leq \beta +\epsilon,
\end{align}
there exists a $d$-dimensional simplex $S=\{x_0,\ldots,x_d\}$ such that $||x_i-x_j||=d_{ij}$ for every $0\leq i<j \leq d$.
\end{corollary}

\begin{proof}
Let $M=(m_{ij})_{0\leq i,j\leq d}$ be the symmetric matrix with zero entries in the main diagonal given by $m_{ij}=d_{ij}^2$ for $i\neq j$. For real numbers $\lambda_0,\ldots, \lambda_d$ satisfying $\lambda_0+\ldots+\lambda_d=0$ and $\lambda_0^2+\ldots+\lambda_d^2=1$ we have that
\begin{align*}
    0=\left(\sum_{i=0}^d \lambda_i\right)^2=1+2\sum_{0\leq i<j\leq d}\lambda_i\lambda_j.
\end{align*}
Hence, 
\begin{align}\label{ceq:eq3}
    \sum_{0\leq i<j\leq d}\lambda_i \lambda_j=-\frac{1}{2}
\end{align}
Thus, by (\ref{ceq:approximate}) and (\ref{ceq:eq3}) we have
\begin{align*}
\left|\left|\sum_{0\leq i<j\leq d}m_{ij}\lambda_i\lambda_j + \frac{\beta^2}{2}\right|\right|=\left|\left|\sum_{0\leq i<j\leq d}(m_{ij}-\beta^2)\lambda_i\lambda_j\right|\right|&\leq \sum_{0\leq i <j \leq d}||m_{ij}-\beta^2||\\
&\leq (d+1)^2(2\epsilon\beta+\epsilon^2)<\frac{\beta^2}{4}.
\end{align*}
This implies that $\sum_{0\leq i<j\leq d}m_{ij}\lambda_i\lambda_j < -\frac{\beta^2}{4}$ and $M$ is strictly of negative type. Therefore, by Theorem \ref{thm:realization} there exists a simplex $S=\{x_0,\ldots,x_d\}$ with $||x_i-x_j||=d_{ij}$ for $0\leq i <j \leq d$.
\end{proof}

Finally, the last auxiliary result shows that any almost regular simplex can be embedded in a brick.

\begin{theorem}[\cites{FR90, MR95}]\label{thm:embedbrick}
For every $\beta, d>0$, there exists a real number $\eta:=\eta(\beta,d)$ such that the following holds. For any simplex $S=\{w_0,\ldots,w_d\}$ satisfying
\begin{align*}
    \beta-\eta\leq ||w_i-w_j||\leq \beta +\eta
\end{align*}
for $0\leq i<j \leq d$, there exists a $\binom{d+1}{2}$-dimensional brick $B$ with $S\subseteq B$.
\end{theorem}

We are now ready to prove Theorem \ref{thm:simplices}.

\begin{proof}[Proof of Theorem \ref{thm:simplices}]
To prove that a simplex is P-Ramsey we will apply again Theorem \ref{thm:robust}, now combined with ideas from \cites{FR90, MR95}. We find convenient to divide the proof in the next four steps:

\vspace{0.2cm}
\noindent \textbf{Step 1:} For a simplex $S=\{x_0,\ldots,x_d\}$ with circumradius $\rho(S)=\rho$, we will find a simplex $S_1=\{y_0,\ldots,y_d\}$ and a small positive real number $\beta$ such that
\begin{align*}
    ||y_i-y_j||^2=||x_i-x_j||^2-\beta
\end{align*}
for all $0\leq i\neq j \leq d$. This implies that the circumradius $\rho(S_1)=\rho'<\rho$.
\vspace{0.2cm}


Let $M=(m_{ij})_{0\leq i,j \leq d}$ be the matrix given by $m_{ij}=||x_i-x_j||^2$. Since $S$ is a simplex, by Theorem \ref{thm:realization} there exists $\gamma>0$ such that 
\begin{align*}
    \sum_{0\leq i<j \leq d}m_{ij}\lambda_i\lambda_j \leq -\gamma
\end{align*}
for all choices of $\lambda_0,\ldots,\lambda_d$ with $\lambda_0+\ldots+\lambda_d=0$ and $\lambda_0^2+\ldots+\lambda_d^2=1$. Set $\beta=\frac{\gamma}{8d^2}$ and let $M'=(m'_{ij})_{0\leq i,j\leq d}$ be the matrix defined by $m'_{ij}=m_{ij}-\beta$ for $i\neq j$ and zero entries in the main diagonal. Since $\beta(d+1)^2\leq 4\beta d^2 <\gamma/2$, then
\begin{align*}
    \sum_{0\leq i<j\leq d}m'_{ij}\lambda_i\lambda_j \leq -\beta\sum_{0\leq i<j\leq d}\lambda_i\lambda_j -\gamma\leq \beta(d+1)^2-\gamma<-\frac{\gamma}{2}<0.
\end{align*}
Consequently, $M'$ is strictly negative, which implies that there exists simplex $S_!=\{y_0,\ldots,y_d\}$ such that
\begin{align}\label{ceq:eq4}
    ||y_i-y_j||^2=m'_{ij}=||x_i-x_j||^2-\beta
\end{align}
for $0\leq i<j \leq d$.

\vspace{0.2cm}
\noindent \textbf{Step 2:} For $\delta \ll \beta$, we find $k$-dimensional vector $c=(c_1,\ldots,c_k)\in \RR^k$ and $S_2=\{z_0,\ldots,z_d\} \subseteq \Spread(c,[n])$ with $||z_i-y_i||<\delta$ for $0\leq i \leq d$. Moreover, 
\begin{align*}
||z_i-z_j||^2=||x_i-x_j||^2-\beta\pm \epsilon
\end{align*}
where $\epsilon:=\epsilon(\beta,d)\rightarrow 0$ as $\delta\rightarrow 0$.
\vspace{0.2cm}

Let $\eta:=\eta(\beta, d)>0$ be the positive real number given by Theorem \ref{thm:embedbrick}, let $\epsilon<\min\{\beta/64d^2, \eta\}$ and let $\delta:=\delta(\eta,\rho)$ be sufficiently small. By Corollary \ref{cor:spread}, there exist integers $n$, $k$, a $k$-dimensional vector $c=(c_1,\ldots,c_k)\in \RR^k$ with $||c||=\rho'$ and a simplex $S_2=\{z_0,\ldots,z_d\}\subseteq \Spread(c,[n])$ such that $||y_i-z_i||<\delta$ for $0\leq i \leq d$. Thus, a triangle inequality gives us that
\begin{align}\label{ceq:anon}
||y_i-y_j||-2\delta<||z_i-z_j||<||y_i-y_j||+2\delta.
\end{align}
Hence, by combining (\ref{ceq:eq4}) and (\ref{ceq:anon})
\begin{align*}
    ||x_i-x_j||^2-\beta+4\delta^2-4\delta||y_i-y_j||<||z_i-z_j||^2<||x_i-x_j||^2-\beta+4\delta^2+4\delta||y_i-y_j||.
\end{align*}
Since $||y_i-y_j||<2\rho'$ and $4\delta^2+8\delta \rho'<\epsilon$ for sufficiently small $\delta$, we have that
\begin{align*}
||x_i-x_j||^2-\beta-\epsilon<||z_i-z_j||^2<||x_i-x_j||^2-\beta+\epsilon.
\end{align*}

\vspace{0.2cm}
\noindent \textbf{Step 3:} We find an almost regular simplex $S_3=\{w_0,\ldots,w_d\}$ satisfying
\begin{align*}
    ||w_i-w_j||^2=||x_i-x_j||^2-||z_i-z_j||^2=\beta\pm \epsilon
\end{align*}
for all $0\leq i\neq j\leq d$. Furthermore, there exists a brick $B$ such that $S_3\subseteq B$.
\vspace{0.2cm}

This is an easy consequence of our preliminary results. By our choice of $\epsilon$, Corollary \ref{cor:almostsimplex} guarantees that there exists simplex $S_3=\{w_0,\ldots,w_d\}$ such that
\begin{align*}
    ||w_i-w_j||^2=||x_i-x_j||^2-||z_i-z_j||^2.
\end{align*}
Moreover, by Theorem \ref{thm:embedbrick}, there exists a $\binom{d+1}{2}$-dimensional brick $B$ such that $S_3\subseteq B$.

\vspace{0.2cm}
\noindent \textbf{Step 4:} We construct a simplex $F\cong S$ such that $F\subseteq B\times \Spread(c,[n])$ and $F\not\subseteq \Spread(c,[n])$ and apply Theorem \ref{thm:robust}
\vspace{0.2cm}

Let $F=\{f_0,\ldots,f_d\}$ be the simplex defined by
\begin{align*}
f_i=w_i*z_i,
\end{align*}
where the symbol $*$ stands for the usual concatenation, i.e., if $a=(a_1,\ldots,a_r)$ and $b=(b_1,\ldots,b_s)$, then $a*b=(a_1,\ldots,a_r,b_1,\ldots,b_s)$. Hence,
\begin{align*}
    ||f_i-f_j||^2=||w_i-w_j||^2+||z_i-z_j||^2=||x_i-x_j||^2
\end{align*}
for every $0\leq i,j \leq d$. Thus, the configuration $F$ is a copy of $S$. Furthermore, by the construction of $F$ we have that 
\begin{align*}
    F\subseteq S_3\times S_2\subseteq B\times \Spread(c,\NN),
\end{align*}
where $B$ is a $\binom{d+1}{2}$-dimensional brick and $\Spread(c,\NN)$ is a robust configuration (Claim \ref{cl:spreadrobust}). Since $\rho(\Spread(c,\NN))=\rho'<\rho=\rho(F)$, we obtain that $F\not\subseteq \Spread(c,\NN)$ and by Theorem \ref{thm:robust} the configuration $F\cong S$ is P-Ramsey.
\end{proof}

\section{Concluding remarks}

The list of known Ramsey configurations is quite limited. Apart of simplices and bricks, the most significant result is due to K\v{r}\'{i}\v{z} \cite{K91} who proved that regular polygons are Ramsey. Unfortunately, our method of robust configurations does not apply here. This leaves us with the following question:

\begin{question}
Are regular polygons P-Ramsey?
\end{question}

Differently from Theorem \ref{thm:densitysimplex}, the proof in \cite{K91} does not provide a density result for regular polygons. Another interesting question would be to determine if such result exists.

\begin{question}
    Let $F$ be a regular polygon. For every $\mu>0$, is there a configuration $Y:=Y(F,\mu)$ such that any set $Z\subseteq Y$ of size $|Z|\geq \mu|Y|$ contains a copy of $F$?
\end{question}

Lastly, another direction of research would be to obtain sharp bounds for the real number $\mu$ in the P-Ramsey definition. It is not hard to show that for a configuration $X$ with $k$ points we cannot take $\mu>\frac{k-1}{k}$. However, our proofs of Theorem \ref{thm:simplices} and \ref{thm:bricks} only gives $\mu=\frac{1}{4}$. It would be interesting to close the gap for simplices.

\begin{question}
    Let $S$ be a $d$-dimensional simplex. What is the largest value of $\mu>0$ such that there exists a configuration $Y$ satisfying properties $(i)$ and $(ii)$ of Definition \ref{def:pramsey}?
\end{question}

\bibliography{literature}

\end{document}